\documentclass{amsart}

\theoremstyle{definition}

\theoremstyle{remark}

\numberwithin{equation}{section}

\begin{document}

\title{FREE SPACES OVER SOME PROPER METRIC SPACES}

\author{A. DALET}

\thanks{The first author was partially supported by PHC Barrande 26516YG}

\address{Laboratoire de Math\'ematiques de Besan\c con, CNRS UMR 6623\\
Universit\'e de Franche-Comt\'e, 16 Route de Gray, 25030 Besan\c con Cedex, FRANCE.}
\curraddr{}
\email{aude.dalet@univ-fcomte.fr}




\newtheorem{Prop}[theorem]{Proposition}
\newtheorem{Coro}[theorem]{Corollary}
\newtheorem{Propr}{Property}

\newcommand{\N}{\mathbb{N}}
\newcommand{\Z}{\mathbb{Z}}
\newcommand{\Q}{\mathbb{Q}}
\newcommand{\C}{\mathbb{C}}
\newcommand{\R}{\mathbb{R}}
\newcommand{\F}[1]{\mathcal{F}(#1)}
\newcommand{\norm}[1]{\|#1\|}


\maketitle

 \newtheorem{thm}{Theorem}[section]
 \newtheorem{cor}[thm]{Corollary}
 \newtheorem{lem}[thm]{Lemma}
 \newtheorem{prop}[thm]{Proposition}
 \theoremstyle{definition}
 \newtheorem{defn}[thm]{Definition}
 \theoremstyle{remark}
 \newtheorem{rem}[thm]{Remark}
 \newtheorem*{ex}{Example}
 \numberwithin{equation}{section}



%
%
%
%
%
%
%
%

\begin{abstract}
We prove that the Lipschitz-free space over a countable proper metric space is isometric to a dual space and has the metric approximation property. We also show that the Lipschitz-free space over a proper ultrametric space is isometric to the dual of a space which is isomorphic to $c_0(\mathbb{N})$.
\end{abstract}

\maketitle



\section{Introduction}

For a pointed metric space $(M,d)$, that is a metric space with an origin $0$, we denote $Lip_0(M)$ the space of Lipschitz real-valued functions on $M$ which vanish at $0$. Endowed with the norm defined by the Lipschitz constant, this space is a Banach space. Moreover, its unit ball is compact for the pointwise topology, hence it is a dual space. 

Let $x\in M$ and define $\delta_x\in Lip_0(M)^*$ as follows: for $f\in Lip_0(M)$, $\delta_x(f)=f(x)$. The Lipschitz-free space over $M$, denoted $\F{M}$, is the closed subspace of $Lip_0(M)^*$ spanned by the $\delta_x$'s: $\F{M}:=\overline{\textrm{span}}\{\delta_x, \ x\in M\}$. Its dual space is isometrically isomorphic to $Lip_0(M)$. 

Lipschitz-free spaces are considered in [20], where they are called Arens-Eells spaces. The notation we use is due to Godefroy and Kalton \cite{GK} where they point out that despite the simplicity of the definition of $\F{M}$, it is not easy to study its linear structure. Although their article was published in 2003, still very little is known about Lipschitz-free spaces. One can check that the Lipschitz-free space over $\mathbb{R}$ is $L_1(\R)$, but Naor and Schechtman proved in \cite{NS} that $\F{\R^2}$ is not isomorphic to a subspace of any $L_1$. Moreover, Godard \cite{Go} proved that $\F{M}$ is isometrically isomorphic to a subspace of an $L_1$-space if and only if $M$ isometrically embeds into an $\R$-tree. We will focus on the notion of approximation property.

\medskip

\noindent A Banach space $X$ has the approximation property (AP in short) if for every positive $\varepsilon$, every $K\subset X$ compact, there exists an operator $T$ on $X$, of finite rank, such that for every $x\in K$, the norm $\|Tx-x\|$ is less than $\varepsilon$.

Let $\lambda \in [1,+\infty)$. The space $X$ has the $\lambda$-bounded approximation property ($\lambda$-BAP) if for every positive $\varepsilon$, every $K\subset X$ compact, there exists an operator $T$ on $X$, of finite rank, such that $\|T\|\leq \lambda$ and for every $x\in K$, the norm $\|Tx-x\|$ is less than $\varepsilon$.

Finally, $X$ has the metric approximation property (MAP) when it has the $1$-BAP.

\medskip

\noindent Godefroy and Kalton \cite{GK} proved that a Banach space has the $\lambda$-BAP if and only if its Lipschitz-free space has the $\lambda$-BAP. Lancien and Perneck\`a \cite{LP}  proved that the Lipschitz-free space over a doubling metric space has the BAP and that $\F{\ell_1}$ has a finite-dimensional Schauder decomposition. H\'ajek and Perneck\`a improved this last result, in \cite{HP} they obtained that $\F{\ell_1}$ has a Schauder basis. However, there are not only positive results, Godefroy and Ozawa \cite{GO} constructed a compact metric space $(K,d)$ such that $\F{K}$ fails the AP. But the author proved in \cite{D} that in the case of countable compact metric spaces, the Lipschitz-free space always has the MAP. In this article, we will prove that the Lipschitz-free space over a countable proper metric space and over a proper ultrametric space is a dual space and has the MAP. More precisely, we show that in the case of a proper ultrametric space, the Lipschitz-free space has an isometric predual which is isomorphic to $c_0(\N)$.

\section{Countable proper metric spaces}

A metric space is said to be proper if every closed ball is compact. 

For a metric space $(M,d)$, we will denote by $B(x,r)$ the open ball of center $x\in M$ and radius $r>0$, and by $\overline{B}(x,r)$ the closed ball.

The space $lip_0(M)$ is the subspace of $Lip_0(M)$ of functions $f$ satisfying: 
$$ \forall \varepsilon>0, \exists \delta>0 :\ d(x,y)<\delta \Rightarrow |f(x)-f(y)|\leq \varepsilon d(x,y)$$

\medskip

\noindent The first result of this section is the following:

\begin{thm}\label{properdual} 
Let $M$ be a countable proper metric space and 
$$S=\left\{ f\in lip_0(M);  \ \lim\limits_{r\rightarrow +\infty}\sup\limits_{\substack{x \textrm{\ or\ } y \notin \overline{B}(0,r)\\ x\neq y}}\cfrac{f(x)-f(y)}{d(x,y)}=0  \right\}.$$
Then, $\F{M}$ is isometrically isomorphic to $S^*$.
\end{thm}

\noindent Before the proof, we need some definitions:

\begin{defn}
\ 
 \begin{enumerate}
  \item Let $X$ be a Banach space. A subspace $F$ of $X^*$ is called separating if $x^*(x)=0$ for all $x^*\in F$ implies $x=0$.
  \item  For $(M,d)$ a pointed metric space, a subspace $F$ of $Lip_0(M)$ separates points uniformly if there exists a constant $c\geq1$ such that for every $x,y\in M$, some $f\in F$ satisfies $\norm{f}_L\leq c$ and $|f(x)-f(y)|=d(x,y)$. 
 \end{enumerate}
\end{defn}
\begin{defn}
Let $X$ be a Banach space. We denote $NA(X)$ the subset of $X^*$ consisting of all linear forms which attain their norm.
\end{defn}

\noindent A result of Petun{\={\i}}n and Pl{\={\i}}{\v{c}}ko \cite{PP} asserts that for a separable Banach space $X$, if a closed subspace $F$ of $X^*$ is separating and is a subset of $NA(X)$, then $X$ is isometrically isomorphic to $F^*$.
To use this result we proceed with a few lemmas about the space $S$.

\begin{lem}\label{NA}
Let $(M,d)$ be proper pointed metric space. The space $S$ is a subspace of $NA(\F{M})$.
\end{lem}

\begin{proof}
Let $f\in S$. We may assume that $f\neq 0$ and take $0<\varepsilon<\frac{\|f\|_L}{2}$. Since 
$$	\lim\limits_{r\rightarrow +\infty}\sup\limits_{\substack{x\textrm{\ or\ }y \notin \overline{B}(0,r)\\ x\neq y}}\cfrac{f(x)-f(y)}{d(x,y)}=0	$$
 there exists $r>0$ such that
$$\sup\limits_{\substack{x\textrm{\ or\ } y \notin \overline{B}(0,r)\\ x\neq y}}\cfrac{|f(x)-f(y)|}{d(x,y)}<\varepsilon.$$
Thus,
$\|f\|_L=\sup\limits_{\substack{x,y\in \overline{B}(0,r)\\ y\neq x}}\cfrac{|f(x)-f(y)|}{d(x,y)}.$

 Because $f\in lip_0(M)$, the set $$\overline{B}^2_{\varepsilon}:=\left\lbrace (x,y)\in \overline{B}(0,r)^2,\ x\neq y,\ |f(x)-f(y)|\geq \varepsilon \  d(x,y)\right\rbrace $$ is compact and we have
\begin{align*}
 \norm{f}_L&=\sup_{\substack{x,y\in \overline{B}(0,r)\\ y\neq x}} \frac{|f(x)-f(y)|}{d(x,y)}=\sup_{(x, y)\in \overline{B}_{\varepsilon}^2} \frac{|f(x)-f(y)|}{d(x,y)}\\
 &=\max_{(x, y)\in \overline{B}_{\varepsilon}^2} \frac{|f(x)-f(y)|}{d(x,y)}.
\end{align*}
Thus, there exist $x\neq y$ such that $\norm{f}_L=\frac{|f(x)-f(y)|}{d(x,y)}$. With $\gamma =\frac{1}{d(x,y)}(\delta_x-\delta_y)$, $\gamma  \in \F{M}$, we obtain $\norm{f}_L=|f(\gamma)| $, with $\norm{\gamma}_{\F{M}}=1$ because $\delta$ is an isometry. Then, $f$ is norm attaining and $S\subset NA(\F{M})$.
\end{proof}

\begin{lem}\label{separation}
Let $(M,d)$ be a proper pointed metric space. If $S$ separates points uniformly, then it is separating.
\end{lem}

\begin{proof}
Using Hahn-Banach theorem it is enough to prove that when $S$ separates points uniformly, it is weak$^{*}$-dense in $Lip_0(M)$.

\medskip

\noindent  We will first prove that the condition $$\lim\limits_{r\rightarrow +\infty}\sup\limits_{\substack{x \textrm{\ or\ } y \notin \overline{B}(0,r)\\ x\neq y}}\cfrac{f(x)-f(y)}{d(x,y)}=0$$ is stable under supremum and infimum between two functions.

Let $f,g\in S$ and $x\neq y$ in $M$ such that $x$ or $y$ doesn't belong to $\overline{B}(0,r)$. We assume that $f(x)\leq g(x)$, the other case is similar. We need to distinguish two cases:
\begin{itemize}
\item if $f(y)\leq g(y)$, then 
\begin{align*} 
\cfrac{\inf(f,g)(x)-\inf(f,g)(y)}{d(x,y)}&=\cfrac{f(x)-f(y)}{d(x,y)}  \end{align*}
and
 \begin{align*}
 \cfrac{\sup(f,g)(x)-\sup(f,g)(y)}{d(x,y)}&=\cfrac{g(x)-g(y)}{d(x,y)}
 \end{align*}
\item if $f(y)\geq g(y)$, then 
\begin{align*} 
\cfrac{f(x)-f(y)}{d(x,y)}\leq\cfrac{\inf(f,g)(x)-\inf(f,g)(y)}{d(x,y)}&=\cfrac{f(x)-g(y)}{d(x,y)}\leq \cfrac{g(x)-g(y)}{d(x,y)} \end{align*}
and
 \begin{align*}
 \cfrac{f(x)-f(y)}{d(x,y)}\leq \cfrac{\sup(f,g)(x)-\sup(f,g)(y)}{d(x,y)}&=\cfrac{g(x)-f(y)}{d(x,y)}\leq \cfrac{g(x)-g(y)}{d(x,y)}
 \end{align*}

So we obtain:
\begin{align*}
\lim\limits_{r\rightarrow +\infty}\sup\limits_{\substack{x \textrm{\ or\ } y \notin \overline{B}(0,r)\\ x\neq y}}\cfrac{\inf(f,g)(x)-\inf(f,g)(y)}{d(x,y)}&=0 \end{align*}
and
 \begin{align*}
\lim\limits_{r\rightarrow +\infty}\sup\limits_{\substack{x \textrm{\ or\ } y \notin \overline{B}(0,r)\\ x\neq y}}\cfrac{\sup(f,g)(x)-\sup(f,g)(y)}{d(x,y)}&=0
\end{align*}
\end{itemize}
\noindent finally $\inf(f,g), \sup (f,g) \in S$.

\medskip

\noindent  Assume that $M$ is a proper pointed metric space such that $S$ separates points uniformly.
Mimicking the proof of Lemma 3.2.3 in \cite{W} we obtain the following:
there exists $b\geq 1$ such that for all $f\in Lip_0(M)$, for all $ A$ finite subset of $M$ containing $0$, we can find $g\in S$ so that $\|g\|_L\leq b\|f\|_L$ and $g_{|A}=f_{|A}$. Finally, one can deduce that the weak$^*$-closure of $S$ is $Lip_0(M)$, so $S$ is separating.
\end{proof}

\noindent  Along the proof of Theorem \ref{properdual} we will need a characterization of compact metric spaces which are countable. First define the Cantor-Bendixon derivation. For a metric space $(M,d)$ we denote:
\begin{itemize}
\item $M'$ the set of accumulation points of $M$.
\item $M^{(\alpha)}=(M^{(\alpha -1)})'$, for a successor ordinal $\alpha$.
\item $M^{(\alpha)}=\bigcap\limits_{\beta<\alpha}M^{(\beta)}$, for a limit ordinal $\alpha$.
\end{itemize}
A compact metric space $(K,d)$ is countable if and only if there is a countable ordinal $\alpha$ such that $K^{(\alpha)}$ is finite.

\renewcommand{\proofname}{Proof of Theorem 2.1:}

\begin{proof}
Note first that the subspace $S$ of $\F{M}^*$ defined previously is closed in $\F{M}^*$, so it follows from Lemmas \ref{NA},  \ref{separation} and from Petun{\={\i}}n and Pl{\={\i}}{\v{c}}ko's result \cite{PP} that we only have to prove that $S$ separates points uniformly.

\medskip

\noindent Ideas are the same as in the proof of Theorem 2.1 in \cite{D} but for sake of completeness we will give all details. Let $M$ be a proper countable metric space, $x,y\in M$ and $a=d(x,y)$. The ball $\overline{B}\left(x, \frac{3a}{2}\right)$ is compact and countable so there exist  a countable ordinal $\alpha_0$, $k_1 \in \N$ and $y_{1}^1, \cdots, y_{k_1}^1\in M$ such that $$\overline{B}\left(x, \frac{3a}{2}\right)^{(\alpha_0)}=\{y_{1}^1, \cdots, y_{k_1}^1\}$$
 We can find $r_1, s_1, t_1$ and $$u_1^1<\cdots<u_{r_1}^1\leq \frac{a}{2}< v_1^1<\cdots<v_{s_1}^1< a\leq w_1^1<\cdots<w_{t_1}^1\leq \frac{3a}{2}$$ such that 
$	\{d(x,y^1_{i}), 1\leq i \leq k_1\}=\{u_1^1,\cdots,u_{r_1}^1,  v_1^1,\cdots,v_{s_1}^1, w_1^1,\cdots,w_{t_1}^1\}	$.
Now set 
$$
u_1=\min\left(
\begin{array}{l}
\left\{\frac{a}{2},\  u^1_1, \ \frac{a}{2}-u_{r_1}^{1}, \ w_1^1-a, \ \frac{3a}{2}-w_{t_1}^1\right\}\backslash \{0\} \\
 \bigcup\left\{u^i_1-u^{i-1}_1 , \  2\leq i \leq r_1\}\cup\{w^i_1-w^{i-1}_1 , \  2\leq i \leq t_1\right\}
 \end{array}  \right)
$$

\noindent and define $\varphi_1 : \left[0,+\infty\right)\rightarrow\left[0,+\infty\right)$ by 

$$
\varphi_1(t) = \left\{
    \begin{array}{ll}
        0 \!\!\!\!\!\!&, \  t \in \left[ 0,\frac{u_1}{4}\right):=U^0_1  \\
        u^i_1\!\!\!\!\!\!&,  \ t\in \left(u^i_1-\frac{u_1}{4},u^i_1+\frac{u_1}{4} \right):=U^i_1 , \  1\leq i \leq r_1\\
	\frac{a}{2}\!\!\!\!\!\!&, \  t\in \left(\frac{a}{2}-\frac{u_1}{4},a+\frac{u_1}{4} \right):=W_1^{0} \ (\textrm{possibly \ } U_1^{r_1}, W^{1}_1\subset W^{0}_1)
\\
	\frac{3a}{2}-w^i_1\!\!\!\!\!\!&,  \ t\in \left(w^i_1-\frac{u_1}{4},w^i_1+\frac{u_1}{4} \right):=W^i_1 , \  1\leq i \leq t_1\\
	0\!\!\!\!\!\!&, \ t\in \left(3\frac{a}{2}-\frac{u_1}{4}, +\infty \right):=W^{t_1+1}_1 \ (\textrm{possibly \ } W^{t_1}_1\subset W^{t_1+1}_1)
    \end{array}
\right.
$$

\noindent and $\varphi_1$ is affine on each interval of $[0,+\infty) \  \backslash \displaystyle \left(\left(\cup_{i=0}^{r_1}U_1^i\right)\cup \left(\cup_{i=0}^{t_1+1}W_1^i\right)\right)$. One can check that $\|\varphi_1\|_L\leq 2$.

 With $f(\cdot)=d(\cdot, x)$, we set $$C_1=f^{-1}\left([0,+\infty) \  \backslash \displaystyle \left(\left(\cup_{i=0}^{r_1}U_1^i\right)\cup \left(\cup_{i=0}^{t_1+1}W_1^i\right)\right)\right).$$

\medskip

\noindent  First, if $C_1$ is finite or empty, define $h(\cdot)=2\left(\varphi_1\circ d(\cdot,x)-\varphi_1\circ d(0,x)\right)$. Then we have $|h(x)-h(y)|=d(x,y)$, $h(0)=0$ and $\|h\|_L\leq 4$. We need to prove that $h\in lip_0(M)$. Set 
$$\delta = \left\{  \begin{array}{ll}
u_1/2, &\textrm{if}\ C_1=\emptyset\\
\cfrac{1}{2}\inf \left( \left\{ u_1, \textrm{sep}(C_1)  \right\}\cup \{\textrm{dist}(z, M\backslash C_1), \ z\in  D_1\} \right), &\textrm{otherwise}
\end{array}\right.$$
where $$\textrm{sep}(C_1)=\inf\{d(z,t), \ z\neq t, \ z,t\in C_1\}$$ and $$D_1=f^{-1}\left(	[0,+\infty) \  \backslash \displaystyle \left(\left(\cup_{i=0}^{r_1}\overline{U_1^i}\right)\cup \left(\cup_{i=0}^{t_1+1}\overline{W_1^i}\right)\right)	\right).$$
Since $C_1$ is finite we have $\textrm{sep}(C_1)>0$. Moreover,  $\textrm{dist}(z, M\backslash C_1)>0$ for $z\in D_1$ and $D_1$ is finite.
Thus we deduce that $\delta>0$.

If follows that every $z\neq t\in M$ such that $d(z,t)\leq \delta$ are not in $D_1$ and there exists $0\leq i\leq r_1$ such that $z,t \in f^{-1}\left(\overline{U_1^i}\right)$ or $0\leq i\leq t_1+1$ such that $z,t \in f^{-1}\left(\overline{W_1^i}\right)$, so the equality $h(z)=h(t)$ holds, i.e. $h\in lip_0(M)$.

Finally, let us prove that $\lim\limits_{r\rightarrow +\infty}\sup\limits_{\substack{x \textrm{\ or\ } y \notin \overline{B}(0,r)\\ x\neq y}}\cfrac{h(x)-h(y)}{d(x,y)}=0$, that is  $h\in S$.

\noindent Let $r>0$ be such that $\overline{B}(x,\frac{3a}{2})\subset \overline{B}(0,r)$ and $z\notin\overline{B}(0,r)$. 

\noindent First if $t\notin \overline{B}\left(x,\frac{3a}{2}\right)$ then $h(z)=h(t)$ and $\cfrac{|h(z)-h(t)|}{d(z,t)}=0$. 

\noindent Secondly if $t\in \overline{B}(x,\frac{3a}{2})$, then 
$$\cfrac{|h(z)-h(t)|}{d(z,t)}=\cfrac{|h(t)|}{d(z,t)}\leq \cfrac{d(x,y)}{d(z,t)}\ \substack{\longrightarrow\\{r\rightarrow +\infty}}\ 0$$
so $h\in S$.

\medskip

\noindent  Assume now that $C_1$ is infinite. It is a subset of $\overline{B}\left(x,\frac{3a}{2}\right)$ thus for every ordinal $\alpha$ we have $C_1^{(\alpha)}\subset \overline{B}\left(x,\frac{3a}{2}\right)^{(\alpha)}$. Moreover, $C_1\cap\overline{B}(x,\frac{3a}{2})^{(\alpha_0)}=\emptyset$ so we have $C_1^{(\alpha_0)}=\emptyset$. Since $C_1$ is compact and countable we can find $\alpha_1<\alpha_0$ such that $C_1^{(\alpha_1)}$ is finite and non empty. 

There exist $k_2\in \N$ and $y_2^1,\cdots, y_2^{k_2}\in C_1$ such that $C_1^{(\alpha_1)}=\{y_2^{1},\cdots, y_2^{k_2}\}$. Then we can find $r_2,t_2\in \N$ and $$u_2^1<\cdots<u_2^{r_2}<\frac{a}{2}-\frac{u_1}{4}, \ \frac{3a}{2}+\frac{u_1}{4}<w_2^1<\cdots<w_2^{t_2}$$ such that 
$$\{d(x,y_2^i)\ ;\ 1\leq i \leq k_2\}=\{u_2^1,\cdots,u_2^{r_2}, w_2^1,\cdots,w_2^{t_2}\}.$$
Set 
$$
u_2=\min\left(
\begin{array}{l}
\left\{u_1, (\frac{a}{2}-\frac{u_1}{4})-u_2^{r_2}, w_2^{1}-(\frac{3a}{2}+\frac{u_1}{4})\right\}\\
 \bigcup\{u_2^i-u_2^{i-1}, 2\leq i\leq r_2\}\cup\{w_2^i-w_2^{i-1}, 2\leq i\leq t_2\} \end{array}  \right)
 $$
and define $\varphi_{2} : \left[0,+\infty\right)\rightarrow\left[0,+\infty\right)$ by 

$$
\varphi_{2}(t) = \left\{
    \begin{array}{ll}
        \varphi_1(t) & , \  t \in \displaystyle \left(\cup_{i=0}^{r_1}U_1^i\right)\cup \left(\cup_{i=0}^{t_1+1}W_1^i\right)  \\
        \varphi_1(u^i_{2}) &, \  t\in \left(u^i_{2}-\frac{u_{2}}{2^{3}},u^i_{2}+\frac{u_{2}}{2^{3}} \right):=U^i_{2} , \ 1\leq i \leq r_{2}\\
         \varphi_1(w^i_{2}) &, \  t\in \left(w^i_{2}-\frac{u_{2}}{2^{3}},w^i_{2}+\frac{u_{2}}{2^{3}} \right):=W^i_{2} , \ 1\leq i \leq t_{2}\\
    \end{array}
\right.
$$
and $\varphi_{2}$ is continuous on $[0,+\infty)$ and affine on each interval of 
$$[0,+\infty)\backslash \left(\left(\cup_{i=0}^{r_1}U_1^i\right)\cup \left(\cup_{i=0}^{t_1+1}W_1^i\right)\cup\left(\cup_{i=1}^{r_2}U_1^i\right)\cup \left(\cup_{i=1}^{t_2}W_1^i\right)\right).$$
It is easy to check that $\|\varphi_2\|_L\leq \frac{8}{3}$. 

\medskip

\noindent  Now we set $C_2=C_1\backslash f^{-1}\left( \left(\cup_{i=1}^{r_2}U_1^i\right)\cup \left(\cup_{i=1}^{t_2}W_1^i\right)\right)$. First if $C_2$ is finite or empty the function $h(\cdot)=2\left(\varphi_2\circ d(\cdot,x)-\varphi_2\circ d(0,x)\right)$ verifies $h(0)=0$,  $|h(x)-h(y)|=d(x,y)$ and $\|h\|_L\leq \frac{16}{3}$. Moreover, if we set
$$\delta = \left\{  \begin{array}{ll}
u_2/2, &\textrm{if}\ C_2=\emptyset\\
\cfrac{1}{2}\min \left( \{ u_2  , \textrm{sep}(C_2)\} \cup \{ \textrm{dist}(z,M\backslash C_{2}), \  z\in D_2 \} \right), &\textrm{otherwise}
\end{array}\right.$$
\noindent where $D_2=C_1\backslash f^{-1}\left( \left(\cup_{i=1}^{r_2}\overline{U_1^i}\right)\cup \left(\cup_{i=1}^{t_2}\overline{W_1^i}\right)\right)$, we obtain that $\delta>0$ and when $z,t\in M$ are such that $d(z,t)\leq \delta$, then $h(z)=h(t)$. So finally $h$ is in $lip_0(M)$. The proof of the fact that $h$ belongs to $S$ is the same as previously. 

\medskip

\noindent  If $C_2$ is infinite we proceed inductively until we get $C_n$ finite, which eventually happens because we construct a decreasing sequence of ordinals.

\bigskip

\noindent The function $h$ we finally obtain verifies $h(0)=0$, $|h(y)-h(x)|=d(x,y)$ and 
$$ \norm{h}_L\leq \displaystyle 2\prod_{j=1}^{n}\left(1+\cfrac{1}{2^j-1} \right)\leq \displaystyle 2\prod_{j=1}^{+\infty}\left(1+\cfrac{1}{2^j-1} \right):=c
$$
\noindent where $c$ does not depend on $x$ and $y$.
Moreover, setting 
$$\delta = \left\{  \begin{array}{ll}
u_n/2, &\textrm{if}\ C_n=\emptyset\\
\cfrac{1}{2}\left(\min \{ u_n  , \textrm{sep}(C_n) \} \cup \{ \textrm{dist}(z,M\backslash C_{n}), \  z\in D_n \} \right), &\textrm{otherwise}
\end{array}\right.$$

\noindent we get $\delta>0$ and if $z,t\in M$ are such that $d(z,t)\leq \delta$, then $h(z)=h(t)$, i.e. $h\in lip_0(M)$. Finally, $h$ still verifies $\lim\limits_{r\rightarrow +\infty}\sup\limits_{\substack{x \textrm{\ or\ } y \notin \overline{B}(0,r)\\ x\neq y}}\cfrac{h(x)-h(y)}{d(x,y)}=0$ so we can conclude that $S$ separates points uniformly and therefore $\F{M}$ is isometrically isomorphic to $S^*$.
\end{proof}
\renewcommand{\proofname}{Proof:}

\noindent  We can now prove the second result of this section:

\begin{thm}
The Lipschitz-free space over a countable proper metric space has the metric approximation property.
\end{thm}

\begin{proof}
A theorem of A. Grothendieck \cite{Gr} asserts that if a separable Banach space is isometrically isomorphic to a dual space and has the AP, then it has the MAP. Thus it follows from Theorem \ref{properdual} that it is enough to prove that for $M$ a countable proper metric space, $\F{M}$ has the BAP.

\medskip

\noindent We need the following result we can deduce from Lemma 4.2 in \cite{K1}:

\noindent For any pointed metric space $M$, define for $N\in \N$, $$A_N=\left(\overline{B}(0,2^{N+1})\backslash B(0,2^{-N-1})\right)\cup\{0\}.$$ Then there exists a sequence of operators $S_N:\F{M}\rightarrow \F{A_N}$, of norm less than $72$, such that for every $\gamma \in \F{M}$  the sequence $(S_N(\gamma))_{N\in \N}$ converges to $\gamma$.

\medskip

\noindent Now let $M$ be a countable and proper metric space. Since every closed ball is compact, the set $A_N$ is countable and compact, for every $N\in \mathbb{N}$. Thus Theorem 3.1 in \cite{D}  asserts that $\F{A_N}$ has the MAP and since for every $N\in \N$, $\F{A_N}$ is separable,
  there exists $R_p^N: \F{A_N}\rightarrow\F{A_N}$ a sequence of operators of finite-rank, so that for every $ \gamma \in \F{A_N}$, $\displaystyle \lim_{p\rightarrow+\infty}R_p^N\gamma =\gamma$ and $\norm{R_p^N}\leq 1$ for every $p\in \N$ (\cite{Pe}, see also Theorem 1.e.13 in \cite{LT}).

Setting $Q_{N,p}=R_p^N\circ S_N$ we deduce that the range of $Q_{N,p}$ is finite dimensional, $\norm{Q_{N,p}}\leq \norm{R_p^N}\norm{S_N}\leq 72$ and for every $\gamma \in \F{M}$, 
$$\lim_{N\rightarrow +\infty}\lim_{p\rightarrow +\infty}R_p^NS_N\gamma=\lim_{N\rightarrow +\infty}S_N\gamma=\gamma.$$ 
Thus $\F{M}$ has the $72$-BAP.

\medskip

\noindent  Finally, we can conclude that $\F{M}$ has the MAP.
\end{proof}

\section{Ultrametric spaces}

A metric space $(M,d)$ is said to be ultrametric if for every $x,y,z\in M$, we have $d(x,z)\leq \max \{d(x,y), d(y,z)\}$. One can easily prove the following useful properties:

\begin{Propr}
For $x,y\in M$ and $r,r'>0$, if $B(x,r)\cap B(y,r')\neq \emptyset$ and $r\leq r'$ then $B(x,r)\subset B(y,r')$.
\end{Propr}
\begin{Propr}
For $x,y\in M$ and $r>0$, if $y\in B(x,r)$, then $B(y,r)=B(x,r)$.
\end{Propr}
\begin{Propr}
For $x,y,z\in M$, if $d(x,y)\neq d(y,z)$ then $$d(x,z)=\max\{d(x,y),d(y,z)\}.$$
\end{Propr}
\begin{Propr}
For every $r>0$ there exists a partition of $M$ in closed balls of radius $r$.
\end{Propr}

\noindent Now let us prove the first result of this section:

\begin{thm}
The Lipschitz-free space over a proper ultrametric space has the metric approximation property.
\end{thm}

\begin{proof}
Let $M$ be a proper ultrametric space and $\tau_p$ the topology of pointwise convergence on $Lip_0(M)$. We will construct a sequence $(L_n)_{n\in \mathbb{N}}$ of operators on $Lip_0(M)$, of norm less than $1$, such that for every $f\in Lip_0(M)$ the sequence $(L_nf)_{n\in \N}$ converges pointwise to $f$. 

\medskip

\noindent  Let $n\in \N$. Because $M$ is ultrametric there exists a partition of $\overline{B}(0,n)$ into balls $\overline{B}(x,\frac{1}{n})$. Moreover, the closed ball $\overline{B}\left(0,n\right)$ is compact then it is possible to find $x_1,\cdots,x_k\in M$ such that $\{\overline{B}\left(x_i,\frac{1}{n}\right)\}_{i=1}^k$ is a finite partition of $\overline{B}(0,n)$. Now define $L_n: Lip_0(M)\rightarrow Lip_0(M)$ as follows:
$$	\begin{array}{rcl}\forall f \in Lip_0(M), \  L_n(f):M&\!\!\!\rightarrow \!\!\!&\mathbb{R}\\x&\!\!\!\mapsto\!\!\!&
\left\{\begin{array}{ll}
f(x_i)\!\!\!&, \textrm{\ where \ }x\in \overline{B}\left(x_i,\frac{1}{n}\right), 1\leq i\leq k	\\
0\!\!\!&,\  x\notin \overline{B}\left(0,n\right)
\end{array}\right.
\end{array}$$

\medskip

\noindent We will first compute the norm of $L_n$. Let $f\in Lip_0(M)$ and $x,y\in M$.
\begin{itemize}
\item If there exists $i\in\{1,\cdots, k\}$ such that $x,y\in \overline{B}\left(x_i,\frac{1}{n}\right)$ then clearly:
$$|L_n(f)(x)-L_n(f)(y)|=0\leq \|f\|_L d(x,y).$$
\item Now assume $x\in \overline{B}\left(x_i,\frac{1}{n}\right)$ and $y\in \overline{B}\left(x_j,\frac{1}{n}\right)$ with $i\neq j$. 

\noindent Remark that because $x\in \overline{B}\left(x_i,\frac{1}{n}\right)$, we have $\overline{B}\left(x_i,\frac{1}{n}\right)=\overline{B}\left(x,\frac{1}{n}\right)$. Furthermore $y\notin \overline{B}\left(x_i,\frac{1}{n}\right)=\overline{B}\left(x,\frac{1}{n}\right)$, so $d(x,y)>\frac{1}{n}$.
\begin{align*}
|L_n(f)(x)-L_n(f)(y)|&=|f(x_i)-f(x_j)|\leq \|f\|_Ld(x_i,x_j)\\
&\leq \|f\|_L\max\{d(x_i,x),d(x_j,x)\}= \|f\|_Ld(x_j,x)\\
&\leq \|f\|_L\max\{d(x_j,y),d(y,x)\}=\|f\|_Ld(x,y)
\end{align*}
\item Finally, for $x\in \overline{B}(0,n)$ and $y\notin \overline{B}(0,n)$, there exists $i\in \{1,\cdots, k\}$ such that $x\in \overline{B}\left(x_i,\frac{1}{n}\right)$. Because $x\in \overline{B}(0,n)$, we have $ \overline{B}(0,n)=\overline{B}(x,n)$ and since $y\notin \overline{B}(0,n)$, we obtain $d(x,y)>n$. Hence
\begin{align*}
|L_n(f)(x)-L_n(f)(y)|=|f(x_i)|&\leq \|f\|_Ld(x_i,0)\leq \|f\|_L\times n \leq \|f\|_Ld(x,y).
\end{align*}

\end{itemize}
Then $\|L_n(f)\|\leq \|f\|_L$ and $\|L_n\|\leq 1$.

\medskip 

\noindent  One can easily prove that $L_n$ is $\tau_p-\tau_p$-continuous and that for $f\in Lip_0(M)$, the sequence $(L_n(f))_{n\in \N}$ pointwise converges to $f$. Then it is the adjoint of an operator $R_n:\F{M}\rightarrow\F{M}$ of norm less than $1$ such that for every $\gamma\in \F{M}$, the sequence $(R_n(\gamma))_{n\in \N}$ weakly-converges to $\gamma$.
Finally, for every $n\in \N$ we have $R_n\left(\F{M}\right)=\textrm{span}\{\delta_{x_i}, 1\leq i \leq k\}$, so the operator $R_n$ is of finite rank. Then, because $\F{M}$ is separable, using convex combinations and a diagonal argument, we can conclude that $\F{M}$ has the MAP \cite{BM}.
\end{proof}

\noindent It is also possible to prove that the Lipschitz-free space over a proper ultrametric space $M$ is a dual space. We will prove first that in the case of $K$ a compact ultrametric space, $\F{K}$ is isometrically isomorphic to $lip_0(K)^*$. 
We will again use the result of Petun{\={\i}}n and Pl{\={\i}}{\v{c}}ko. Note that Theorem 3.3.3 in \cite{W} provides an alternative approach. 

\medskip

\noindent  Before stating the result let us introduce the notion of $\R$-trees and some background about its link with ultrametric spaces:

\begin{defn}
A metric space $(T,d)$ is said to be an $\R$-tree when the two following conditions hold:
\begin{enumerate}
\item for every $a,b$ in $T$, there exists a unique isometry $\phi:[0,d(a,b)]\rightarrow T$ such that $\phi (0)=a$ and $\phi (d(a,b))=b$.
\item any continuous and one-to-one mapping $\varphi : [0,1]\rightarrow T$ has same range as the isometry $\phi$ associated to the points $a=\varphi(0)$ and $b=\varphi (1)$. 
\end{enumerate}
\end{defn}

\noindent{\it Background.} P. Buneman proved in \cite{B} that the 4-points property is a characterization of subsets of $\R$-trees, where a metric space $(M,d)$ has the 4-points property if for every $x, y, z$ and $t$ in $M$ we have:
$$d(x,y)+d(z,t)\leq \max \left\{d(x,z)+d(y,t)\ ,\ d(x,t)+d(y,z)\right\}.$$

\noindent In particular  any ultrametric space $(M,d)$ has the 4-points property.

It is proved in \cite{M} by Matou\v{s}ek that, for a subspace $M$ of a tree $T$, it is possible to find a linear extension operator from $Lip_0(M)$ to $Lip_0(T)$ which is bounded. In particular $\F{M}$ is complemented in $\F{T}$.

Moreover, Godard proved in \cite{Go} that the Lipschitz-free space over an $\R$-tree is an $L_1$-space. 

In conclusion if $M$ is a ultrametric space, its Lipschitz-free space is complemented into an $L_1$-space.

\begin{thm}\label{compactultra}
If $(K,d)$ is a compact ultrametric space, then $\F{K}$ is isometrically isomorphic to $lip_0(K)^*$ and $lip_0(K)$ is isomorphic to $c_0(\N)$.
\end{thm}

\begin{proof}
It is proved in \cite{D} that for a compact metric space $(K,d)$, the space $lip_0(K)$ is a subset of $NA\left(\F{K}\right)$ and it is separating as soon as it separates points uniformly.

Let $(K,d)$ be a compact ultrametric space. To obtain the first part of the result it is enough to prove that $lip_0(K)$ separates points uniformly.

\noindent Let $x,y\in K$, set $a=d(x,y)$ and define $h:K\rightarrow \mathbb{R}$ as follows:
$$\forall z\in K, \ h(z)=d(x,y)\left(\mathbf{1}_{B(x,a/2)}(z)-\mathbf{1}_{B(x,a/2)}(0)\right)$$
where $\mathbf{1}_{B(x,a/2)}$ is the characteristic function of the open ball $B(x,a/2)$.

\noindent Then we have $h(0)=0$ and $|h(x)-h(y)|=d(x,y)$. We will compute the Lipschitz-constant of $h$: 

\noindent If $z,t$ are both in $B(x,a/2)$ or both outside $B(x,a/2)$, then $$|h(z)-h(t)|=0\leq 2d(z,t).$$

\noindent Take $z\in B(x,a/2)$ and $t\notin B(x,a/2)$, then 
$$d(z,t)=\max\{d(x,z), d(x,t)\}=d(x,t)\geq\cfrac{a}{2}=\cfrac{d(x,y)}{2}=\cfrac{|h(z)-h(t)|}{2}.$$

\noindent Hence the function $h$ is $2$-Lipschitz. 

\medskip

\noindent  To conclude we need to prove that $h\in lip_0(K)$. We will see that $\delta=\frac{a}{2}$ holds for every $\varepsilon$: 

\noindent Let $z,t\in K$ such that $d(z,t)<\frac{a}{2}$.

\noindent First, if $z\in B(x,a/2)$ then $B(x,a/2)=B(z,a/2)$ and because $d(z,t)<\frac{a}{2}$ we have $t\in B(x,a/2)$ and $h(z)=h(t)$.

\noindent Secondly, if $z\notin B(x,a/2)$ then $t$ cannot be in $B(x,a/2)$ and $h(z)=h(t)$.

\medskip

\noindent  This proves that $h$ is in $lip_0(K)$ so this space separates points uniformly and therefore this concludes the proof of the fact that $\F{K}$ is the dual space of $lip_0(K)$.

\medskip
 
\noindent  A result due to D.R. Lewis and C. Stegall \cite{LS} asserts that if a separable dual space is complemented in $L_1$, then it is isomorphic to $\ell_1(\N)$. So it follows from the background before the theorem that $\F{K}$ is isomorphic to $\ell_1(\N)$.

\bigskip
 
\noindent  Finally, Theorem 6.6 in \cite{K1} asserts that for a compact metric space $K$, the space $lip_0(K)$ is isomorphic to a subspace of $c_0(\N)$. Moreover, its dual is isomorphic to $\ell_1(\N)$, then Corollary 2 in \cite{JZ1} implies that $lip_0(K)$ is isomorphic to $c_0(\N)$.  \end{proof}

\noindent More generally for a proper ultrametric space we have the following:

\begin{thm}\label{Thmultrap}
Let $M$ be a proper ultrametric space and 
$$S=\left\{ f\in lip_0(M);  \ \lim\limits_{r\rightarrow +\infty}\sup\limits_{\substack{x \textrm{\ or\ } y \notin \overline{B}(0,r)\\ x\neq y}}\cfrac{f(x)-f(y)}{d(x,y)}=0  \right\}.$$
Then $\F{M}$ is isometrically isomorphic to $S^*$ and $S$ is isomorphic to $c_0(\N)$.
\end{thm}

\begin{proof}
 It is possible to adapt the proof of Theorem 6.6 in \cite{K1} to obtain that the space $S$ is isomorphic to a subspace of $c_0(\N)$:

\begin{lem}\label{Lemma}
Let $M$ be a proper metric space. Then for any $\varepsilon >0$, the space $S$ is $(1+\varepsilon)$-isometric to a subspace of $c_0(\N)$.
\end{lem}

\begin{proof}
Assume $\varepsilon <1$ and consider the space $M\times M$ with the metric :
$$d((x_1,x_2),(y_1,y_2))=\max \left\{d(x_1,y_1), d(x_2,y_2)\right\}.$$
For every $j\in \N$ and $k\in \Z$ we consider the compact set 
$$C_{j,k}=\left\{(x_1,x_2)\in M\times M \ ;\ d(0,x_1)\leq 2^j \textrm{\ and \ } 2^k\leq d(x_1,x_2)\leq 2^{k+1}\right\}$$
and $F_{j,k}$ a finite $2^{k-3}\varepsilon$-net of $C_{j,k}$. Then $F:=\bigcup\limits_{\substack{j\in \N\\ k\in \Z}}F_{j,k}$ is countable.

We now define 
$$\begin{array}{rcl}T:S&\rightarrow &c_0(F)\\f&\mapsto&\left(\cfrac{f(x_1)-f(x_2)}{d(x_1,x_2)}\right)_{(x_1,x_2)\in F}.
\end{array}$$

\medskip

\noindent  Justify first that $Tf\in c_0(F)$ for $f\in S$:

Let $\alpha>0$. 

Because $f\in S$, in particular $f\in lip_0(M)$ and there exists $K \in \N$ such that for every $k\leq -K$, if $d(x_1,x_2)\leq 2^{k+1}$ then $\cfrac{|f(x_1)-f(x_2)|}{d(x_1,x_2)}\leq \alpha$. Thus for every $j\in \N$, every $k\leq -K$ and $(x_1,x_2)\in C_{j,k}$, we have $\cfrac{|f(x_1)-f(x_2)|}{d(x_1,x_2)}\leq \alpha$.

Moreover,  $\lim\limits_{r\rightarrow +\infty}\sup\limits_{\substack{x \textrm{\ or\ } y \notin \overline{B}(0,r)\\ x\neq y}}\cfrac{f(x)-f(y)}{d(x,y)}=0$, thus\! there\! exists\! $R\!>\!0$ such that $\forall r\geq R$, $\forall x\notin \overline{B}(0,r)$, $y\in M$, we have $\cfrac{|f(x)-f(y)|}{d(x,y)}\leq \alpha$. 

Let $N\in \N$ be such that $2^n\geq 2R$, $\forall n\geq N$.


If $(x_1,x_2)\in C_{j,k}$ with $j\geq N$ we clearly have $\cfrac{|f(x_1)-f(x_2)|}{d(x_1,x_2)}\leq \alpha$

 Assume now $(x_1,x_2)\in C_{j,k}$ with $k> N$ and $j\leq N$, then $$d(0,x_2)\geq d(x_1,x_2)-d(0,x_1)\geq2^k-R> R$$ that is $x_2\notin \overline{B}(0,R)$ and $\cfrac{|f(x_1)-f(x_2)|}{d(x_1,x_2)}\leq \alpha$.

\medskip

Finally, we obtain that $Tf\in c_0(F)$, for every $f\in S$.

\medskip

\noindent  Clearly $\|T\|\leq 1$. We will now show that $\|f\|_L\leq (1+\varepsilon)\|Tf\|_{\infty}$:

\noindent Let $y_1\neq y_2\in M$. There exists $j\in \N$ and $k\in \Z$ such that $(y_1,y_2)\in C_{j,k}$ and $(x_1,x_2)\in F_{j,k}$ such that $d((y_1,y_2),(x_1,x_2))\leq 2^{k-3}\varepsilon$. Then 
\begin{align*}
d(y_1,y_2)&\geq d(x_1,x_2)-d(x_1,y_1)-d(x_2,y_2)\geq d(x_1,x_2)-2^{k-2}\varepsilon\\
&\geq d(x_1,x_2)\left(1-\cfrac{\varepsilon}{4}\right).
\end{align*}
\noindent Let $f\in S$, 
\begin{align*}
\cfrac{|f(y_1)-f(y_2)|}{d(y_1,y_2)}&\leq \cfrac{|f(x_1)-f(x_2)|}{d(y_1,y_2)} + \cfrac{d(x_1,y_1)+d(x_2,y_2)}{d(y_1,y_2)}\|f\|_L\\
&\leq \cfrac{|f(x_1)-f(x_2)|}{d(y_1,y_2)} + \cfrac{\varepsilon}{4}\|f\|_L\\
&\leq \left(1-\cfrac{\varepsilon}{4}\right)^{-1}\cfrac{|f(x_1)-f(x_2)|}{d(x_1,x_2)}+\cfrac{\varepsilon}{4}\|f\|_L\\
&\leq  \left(1-\cfrac{\varepsilon}{4}\right)^{-1}\|Tf\|_{\infty}+\cfrac{\varepsilon}{4}\|f\|_L
\end{align*}
Finally, $\|Tf\|_{\infty}\leq\|f\|_L\leq (1+\varepsilon)\|Tf\|_{\infty}$ and one can conclude that $S$ is $(1+\varepsilon)$-isometric to a subspace of $c_0(\N)$.
\end{proof}

\noindent We now conclude the proof of Theorem \ref{Thmultrap}. 
We previously proved that in the case of a proper metric space, the space $S$ is a subspace of $NA(\F{M})$ and it is separating as soon as it separates points uniformly. Therefore in order to use Petun{\={\i}}n and Pl{\={\i}}{\v{c}}ko's result \cite{PP} (see also \cite{G1}) we only need to prove that, in the case of proper ultrametric space, the space $S$ separates points uniformly.

For given $x,y\in M$, the function $h$ defined as in proof of Theorem \ref{compactultra} satisfies $h\in lip_0(M)$, $|h(x)-h(y)|=d(x,y)$ and its Lipschitz constant does not depend on $x$ and $y$.

Let $r>0$ be such that $B(x,a/2)\subset B(0,r)$, with $a=d(x,y)$. We may and do assume that $d(z,0)>r$. 

\noindent First if $t\in \overline{B}(x,\frac{a}{2})$, then
$$	\cfrac{|h(z)-h(t)|}{d(z,t)}=\cfrac{d(x,y)}{d(z,t)} \substack{\longrightarrow\\{r\rightarrow +\infty}}\ 0.	$$
Secondly if $t\notin\overline{B}(x,\frac{a}{2})$, then
$$	\cfrac{|f(z)-f(t)|}{d(z,t)}= 0.	$$

\noindent Finally, we have $h\in S$, then $S$ separates points uniformly. We can conclude that $S^*$ is isometrically isomorphic to $\F{M}$.

\bigskip

\noindent  The second part of the proof follows the same line than the last part of the proof of Theorem \ref{compactultra}.
\end{proof}

\medskip

\begin{rem} B.R. Kloeckner proved in \cite{Kl} that the Wasserstein space of a compact ultrametric space is affinely isometric to a convex subset of $\ell_1(\N)$.
\end{rem}


\subsection*{Acknowledgments}
The author would like to thank Gilles Godefroy and Gilles Lancien for useful conversations and comments, and Henri Lombardi for asking the question of the properties of Lipschitz-free spaces over ultrametric spaces. Finally, the author is grateful to the referee for useful comments and suggestions which permitted to improve the results and the presentation of this paper.

\end{document}